\documentclass[reqno]{amsart}

\usepackage{amsmath,amssymb,amsthm,amscd}
\usepackage[mathscr]{eucal}
\usepackage{color}

\newtheorem{theorem}{Theorem}
\newtheorem{lemma}[theorem]{Lemma}

\newtheorem{proposition}[theorem]{Proposition}
\newtheorem{corollary}[theorem]{Corollary}

\theoremstyle{definition}

\newcommand{\Q}{{\mathbb Q}}

\newcommand{\Z}{{\mathbb Z}}

\newcommand{\eb}{\boldsymbol{e}}

\newcommand{\jb}{\boldsymbol{j}}

\newcommand{\lb}{\boldsymbol{l}}
\newcommand{\mb}{\boldsymbol{m}}
\newcommand{\ob}{\boldsymbol{o}}
\newcommand{\ot}{\otimes}
\newcommand{\pair}[1]{\langle{#1}\rangle}
\newcommand{\U}{\mathbf{U}}
\newcommand{\vep}{\varepsilon}
\newcommand{\wb}{\boldsymbol{w}}

\begin{document}

\title[]{Solution to the reflection equation related \\ to the $\iota$quantum group of type AII}

\author{HIROTO KUSANO and MASATO OKADO}

\address{Department of Mathematics, Osaka City University, Osaka 558-8585, Japan}

\begin{abstract}
A solution to the reflection equation associated to a coideal subalgebra of 
$U_q(A_{2n-1}^{(1)})$ of type AII in the symmetric tensor representations is presented. 
If parameters of the coideal subalgebra are suitably chosen, 
the $K$ matrix does not depend on the quantum parameter $q$ and still agrees with
a solution in \cite{KOYa} at $q=0$.
\end{abstract}

\maketitle

\section{Introduction}

Reflection equation assures the integrability in one-dimensional quantum systems or 
two-dimensional statistical models with boundaries. In the context of quantum 
integrability, it is an equation involving two kinds of linear operators, called quantum
$R$ and $K$ matrices, on the twofold tensor product of vector spaces. 
The mathematical framework to construct its solution lies in considering a pair of
a quantum group and its coideal subalgebra. They are called a quantum symmetric 
pair \cite{L} or an $\iota$quantum group \cite{BW1} and known to be classified by
Satake diagrams \cite{L,Ko}. In such a situation, $R$ and $K$ matrices 
contain the quantum parameter $q$. Moreover, if the representations
have crystal bases in the sense of Kashiwara \cite{K}, one can take the limit where
$q$ goes to 0, and we obtain bijections between sets that still satisfy a combinatorial 
version of the reflection equation.

In \cite{KOYa}, from the motivation of constructing a so-called box-ball system
with boundary, we found three solutions of the combinatorial $K$ matrix where
the combinatorial $R$ matrix in the reflection equation comes from the crystal basis 
of the symmetric tensor representation of the quantum affine algebra of type $A$.
See (2.10)-(2.12) of \cite{KOYa}. They were called ``Rotateleft'', ``Switch$_{12}$''
and ``Switch$_{1n}$''.
However, their quantum versions, namely, solutions of quantum $K$ matrices, were 
not found for a long time. Only recently, in \cite{KOYo} the solution corresponding
to ``Rotateleft'' were found.
The purpose of this note is to find the origin of the other two solutions
``Switch$_{12}$'' and ``Switch$_{1n}$'' from the list of $\iota$quantum groups.
The correct one was found to be the affine version
of type AII. See e.g.~\cite{L,Ko,W}. Rather surprisingly, if we choose parameters in
our $\iota$quantum group suitably, the $K$ matrices does not depend on $q$,
although the $R$ matrices do.

There are many $\iota$quantum groups other than affine type AII which we
dealt with in this note, and there also exists a notion of the universal $K$matrix
\cite{Ko,BW1,BW2} as with the universal $R$ matrix of a quantum group. 
We hope to report more solutions of the reflection equation that become combinatorial
upon taking the limit $q\to0$ in near future.

\section{$U_q(A^{(1)}_{2n-1})$ and relevant $R$ matrices} \label{sec:Uq}

\subsection{\mathversion{bold}$U_q(A^{(1)}_{2n-1})$ and relevant representations}
Let $\U = U_q(A^{(1)}_{2n-1})$ be the Drinfeld-Jimbo quantum affine algebra
(without the derivation operator). In this note, we assume $n\ge2$. $\U$ is 
generated by $e_i, f_i, k^{\pm 1}_i\, (i \in \Z_{2n})$ obeying the relations
\begin{equation}\label{hn}
\begin{split}
&k_i k^{-1}_i = k^{-1}_i k_i = 1,\;\; [k_i, k_j]=0,\;\;
k_ie_jk^{-1}_i = q^{a_{ij}}e_j,\;\;
k_if_jk^{-1}_i = q^{-a_{ij}}f_j,\;\;
[e_i, f_j]=\delta_{ij}\frac{k_i-k^{-1}_i}{q-q^{-1}},\\
&\sum_{\nu=0}^{1-a_{ij}}(-1)^\nu
e^{(1-a_{ij}-\nu)}_i e_j e_i^{(\nu)}=0,
\quad
\sum_{\nu=0}^{1-a_{ij}}(-1)^\nu
f^{(1-a_{ij}-\nu)}_i f_j f_i^{(\nu)}=0\;\;(i\neq j),
\end{split}
\end{equation}
where $e^{(\nu)}_i = e^\nu_i/[\nu]!, \,
f^{(\nu)}_i = f^\nu_i/[\nu]!$
and 
$[m]! = \prod_{j=1}^m [j]$.
The Cartan matrix $(a_{ij})_{i,j \in \Z_{2n}}$ is given by 
$a_{ij}= 2\delta_{i,j}-\delta_{i,j+1}-\delta_{i,j-1}$. 
It is well known that $\U$ is a Hopf algebra. 
We employ the coproduct $\Delta$ of the form 
\begin{equation}
\Delta (k^{\pm 1}_i) = k^{\pm 1}_i\otimes k^{\pm 1}_i,\quad
\Delta (e_i) = e_i\otimes 1 + k_i \otimes e_i,\quad
\Delta (f_i)  = f_i\otimes k^{-1}_i + 1 \otimes f_i.
\end{equation}

We will be concerned with the two irreducible representations 
of $\U$ labeled with a positive integer $l$:
\begin{align}
\pi_{l,x}: & \;U_q \rightarrow \mathrm{End}(V_{l,x}),
\quad V_{l,x}= \bigoplus_{\alpha \in B_l} \Q(q)v_\alpha,
\label{obt1}\\
\pi^*_{l,x}: & \; U_q \rightarrow \mathrm{End}(V^*_{l,x}),
\quad V^*_{l,x} = \bigoplus_{\alpha \in B_l} \Q(q)v^*_\alpha,
\label{obt2}
\end{align}
where $x$ is a spectral parameter in $\Q(q)$ and
\begin{align}
B_l = \{\alpha=(\alpha_1, \ldots, \alpha_{2n}) \in \Z_{\ge0}^{2n} \mid |\alpha|= l\}.
\end{align}
Here $|\alpha|=\sum_{i=1}^{2n}\alpha_i$.
The actions of the generators of $\U$ on these representations are given by
\begin{alignat}{2}
e_iv_\alpha &= x^{\delta_{i,0}}  [\alpha_{i+1}] 
v_{\alpha+\eb_i-\eb_{i+1}},
& \qquad\quad
e_iv^*_\alpha &= x^{\delta_{i,0}}
[\alpha_i]v^*_{\alpha-\eb_i+\eb_{i+1}},
\label{otm1}\\
f_iv_\alpha &= x^{-\delta_{i,0}} [\alpha_{i}] 
v_{\alpha-\eb_i+\eb_{i+1}},
&
f_iv^*_\alpha &= x^{-\delta_{i,0}} [\alpha_{i+1}]
v^*_{\alpha+\eb_i-\eb_{i+1}},
\label{otm2}\\
k_iv_\alpha &= q^{\alpha_i-\alpha_{i+1}}v_\alpha, 
&
k_iv^*_\alpha &= q^{-\alpha_i+\alpha_{i+1}}v^*_\alpha.
\label{otm3}
\end{alignat}
Here $\eb_i$ is the $i$-th standard basis vector and 
the index $j$ of the Chevalley generators or $\alpha$ should be understood as elements
of $\Z_{2n}$. $V_{l,x}$ is the $l$-th symmetric tensor representation of $\U$. 
$V^*_{l,x}$ is constructed on the dual space of $V_{l,x}$ by using the 
anti-automorphism $*$ of $\U$ defined on the generators as
\[
e_i^*=e_i,\quad f_i^*=f_i,\quad k_i^*=k_i^{-1},
\]
and by defining actions on $V^*_{l,x}$ as $\pair{uv^*,v}=\pair{v^*,u^*v}$ for 
$u\in\U,v\in V_{l,x},v^*\in V^*_{l,x}$. Our basis $\{v^*_\alpha\}$ of $V^*_{l,x}$ is
changed from the dual basis of $\{v_\alpha\}$ by multiplying $\prod_j[\alpha_j]!^{-1}$
on each dual basis vector, so it turns out that 
when $x=1$ both $\{v_\alpha\}$ and $\{v^*_\alpha\}$ are upper crystal bases \cite{K}.
At $q=0$, the former gives the crystal $B_l$ and 
the latter its dual $B^\vee_l$ in \cite{KOYa}.

\subsection{$R$ matrices}

We consider the following three $R$ matrices $R,R^*,R^{**}$ that are defined as 
intertwiners between the tensor product representations below.
\begin{alignat}{2}
R(x/y):& \; V_{l,x} \otimes V_{m,y} 
\rightarrow V_{m,y} \otimes V_{l,x},
&\quad
(\pi_{m,y} \otimes \pi_{l,x})\Delta(u)R(x/y)
&= R(x/y)(\pi_{l,x} \otimes \pi_{m,y})\Delta(u),
\label{LL1}\\
R^*(x/y):& \; V^*_{l,x} \otimes V_{m,y} 
\rightarrow V_{m,y} \otimes V^*_{l,x},
&\quad
(\pi_{m,y} \otimes \pi^*_{l,x})\Delta(u)R^*(x/y)
&= R^*(x/y)(\pi^*_{l,x} \otimes \pi_{m,y})\Delta(u),
\label{LL2}\\
R^{* *}(x/y):
& \; V^*_{l,x} \otimes V^*_{m,y} 
\rightarrow V^*_{m,y} \otimes V^*_{l,x},
&\quad 
(\pi^*_{m,y} \otimes \pi^*_{l,x})\Delta(u)R^{**}(x/y)
&= R^{**}(x/y)(\pi^*_{l,x} \otimes \pi^*_{m,y})\Delta(u),
\label{LL3}
\end{alignat}
where $u\in\U$. They satisfy the Yang-Baxter equations:
\begin{align}
(1\otimes R(x))(R(xy) \otimes 1)(1\otimes R(y))
&=
(R(y)\otimes 1)(1\otimes R(xy))(R(x) \otimes 1),
\label{rb4}\\
(1\otimes R^*(x))(R^*(xy) \otimes 1)(1\otimes R(y))
&=
(R(y)\otimes 1)(1\otimes R^*(xy))(R^*(x) \otimes 1),
\\
(1\otimes R^{* *}(x))(R^*(xy) \otimes 1)(1\otimes R^*(y))
&=
(R^*(y)\otimes 1)(1\otimes R^*(xy))(R^{**}(x) \otimes 1),
\\
(1\otimes R^{* *}(x))(R^{* *}(xy) \otimes 1)(1\otimes R^{* *}(y))
&=
(R^{* *}(y)\otimes 1)(1\otimes R^{* *}(xy))(R^{**}(x) \otimes 1).
\label{rb5}
\end{align}

\section{Reflection equation and its solution}

\subsection{Coideal subalgebra}

We consider two coideal subalgebras $\U^\iota_\vep$ ($\vep=0,1$) of $\U$.
Set $I=\{0,1,\ldots,2n-1\}$. An element of $I$ is considered to correspond to 
a vertex of the Dynkin diagram of $A^{(1)}_{2n-1}$. In view of this, we identify $I$
with $\Z_{2n}$. For each $\vep=0,1$, set
\[
I_\circ=\{\vep,2+\vep,\ldots,2n-2+\vep\},\quad I_\bullet=I\setminus I_\circ.
\]
We define two subalgebras ${\bf U}^\iota_\vep$ of $\U$ for $\vep=0,1$.
Each one is generated by
$e_i,f_i,k_i\,(i\in I_\bullet),b_i\,(i\in I_\circ)$ where
\begin{align*}
b_i&=f_i+\gamma_i T_{w_\bullet}(e_i)k_i^{-1},\\
T_{w_\bullet}(e_i)&=e_{i+1}e_{i-1}e_i-q^{-1}(e_{i+1}e_ie_{i-1}+e_{i-1}e_ie_{i+1})+q^{-2}e_ie_{i-1}e_{i+1}.
\end{align*}
Here $\gamma_i$ is a constant. Then, we have 

\begin{proposition} \label{prop:eb=be}
For $i\in I_\circ$, $e_{i\pm1}b_i=b_ie_{i\pm1}$.
\end{proposition}

\begin{table}[h]
\setlength{\unitlength}{0.15in}
\hspace{-4cm}
\begin{minipage}{7cm}
\begin{picture}(24,4)
\put(7.95,1){\makebox(0,0)[c]{$\bullet$}}
\put(10,1){\makebox(0,0)[c]{$\circ$}}
\put(12.1,1){\makebox(0,0)[c]{$\bullet$}}
\put(16.75,1){\makebox(0,0)[c]{$\cdots$}}
\put(17.5,1){\makebox(0,0)[c]{$\cdots$}}
\put(18.25,1){\makebox(0,0)[c]{$\cdots$}}
\qbezier(8.1,1.25)(10.4,3.4)(14.2,3.5)
\qbezier(14.8,3.5)(18.45,3.4)(20.45,1.25)
\put(8.2,1){\line(1,0){1.5}}
\put(10.3,1){\line(1,0){1.5}}
\put(12.45,1){\line(1,0){1.5}}
\put(14.6,1){\line(1,0){1.4}}
\put(18.9,1){\line(1,0){1.4}}
\put(14.3,1){\makebox(0,0)[c]{$\circ$}}
\put(14.5,3.5){\makebox(0,0)[c]{$\circ$}}
\put(20.6,1){\makebox(0,0)[c]{$\bullet$}}
\put(14.5,2.8){\makebox(0,0)[c]{\tiny $0$}}
\put(14.3,0.3){\makebox(0,0)[c]{\tiny $4$}}
\put(8,0.3){\makebox(0,0)[c]{\tiny $1$}}
\put(10,0.3){\makebox(0,0)[c]{\tiny $2$}}
\put(12.1,0.3){\makebox(0,0)[c]{\tiny $3$}}
\put(20.6,0.3){\makebox(0,0)[c]{\tiny $2n-1$}}
\end{picture}
\end{minipage}
\vspace{5mm}
\begin{minipage}{7cm}
\begin{picture}(24,4)
\put(7.95,1){\makebox(0,0)[c]{$\circ$}}
\put(10,1){\makebox(0,0)[c]{$\bullet$}}
\put(12.1,1){\makebox(0,0)[c]{$\circ$}}
\put(16.75,1){\makebox(0,0)[c]{$\cdots$}}
\put(17.5,1){\makebox(0,0)[c]{$\cdots$}}
\put(18.25,1){\makebox(0,0)[c]{$\cdots$}}
\qbezier(8.1,1.25)(10.4,3.4)(14.2,3.5)
\qbezier(14.8,3.5)(18.45,3.4)(20.45,1.25)
\put(8.2,1){\line(1,0){1.5}}
\put(10.3,1){\line(1,0){1.5}}
\put(12.45,1){\line(1,0){1.5}}
\put(14.6,1){\line(1,0){1.4}}
\put(18.9,1){\line(1,0){1.4}}
\put(14.3,1){\makebox(0,0)[c]{$\bullet$}}
\put(14.5,3.5){\makebox(0,0)[c]{$\bullet$}}
\put(20.6,1){\makebox(0,0)[c]{$\circ$}}
\put(14.5,2.8){\makebox(0,0)[c]{\tiny $0$}}
\put(14.3,0.3){\makebox(0,0)[c]{\tiny $4$}}
\put(8,0.3){\makebox(0,0)[c]{\tiny $1$}}
\put(10,0.3){\makebox(0,0)[c]{\tiny $2$}}
\put(12.1,0.3){\makebox(0,0)[c]{\tiny $3$}}
\put(20.6,0.3){\makebox(0,0)[c]{\tiny $2n-1$}}
\end{picture}
\end{minipage}
\caption{Satake diagrams of $\U^\iota_0$ and $\U^\iota_1$}
\end{table}
The following fact is well known. See \cite{L,Ko,W} for instance.

\begin{proposition}
${\bf U}^\iota_\vep$ is a right coideal subalgebra of $\U$. Namely, we have $\Delta(\U^\iota_\vep)\subset \U^\iota_\vep\ot\U$.
\end{proposition}

We also use the following result later.

\begin{lemma} \label{lem:bi}
For $i\in I_\circ$, the action of $b_i$ on $V_{l,x}$ or $V^*_{l,x}$ is given by
\begin{align*}
b_iv_{\alpha}&=x^{-\delta_{i,0}}[\alpha_i]v_{\alpha-\eb_i+\eb_{i+1}}
-x^{\delta_{i,0}+\delta_{i,1}+\delta_{i,-1}}q^{-1}\gamma_i[\alpha_{i+2}]
v_{\alpha+\eb_{i-1}-\eb_{i+2}},\\
b_iv^*_{\alpha}&=x^{\delta_{i,0}}[\alpha_{i+1}]v^*_{\alpha+\eb_i-\eb_{i+1}}
-x^{-\delta_{i,0}-\delta_{i,1}-\delta_{i,-1}}q^{-1}\gamma_i[\alpha_{i-1}]
v^*_{\alpha-\eb_{i-1}+\eb_{i+2}}.
\end{align*}
\end{lemma}

\subsection{$K$ matrix and the reflection equation}

For each $\vep=0,1$,
consider a linear map $K(x):V_{l,x} \rightarrow V^*_{l,x^{-1}}$ 
satisfying
\begin{align} \label{intertwine}
 K(x)\pi_{l,x}(a) = \pi^*_{l, x^{-1}}(a) K(x) \quad 
\text{for any }a \in {\bf U}^\iota_\vep.
\end{align}
To describe the solution, we introduce a particular permutation $\sigma^{(\vep)}$ 
of entries of $\alpha$ for $\vep=0,1$.  $\sigma^{(\vep)}$ switches $\alpha_{i-1}$ and
$\alpha_{i}$ whenever $i\equiv\vep$ (mod 2).
For instance, when $n=3$ we have
\[
\sigma^{(0)}(\alpha)=(\alpha_2,\alpha_1,\alpha_4,\alpha_3,\alpha_6,\alpha_5),\quad
\sigma^{(1)}(\alpha)=(\alpha_6,\alpha_3,\alpha_2,\alpha_5,\alpha_4,\alpha_1).
\]

\begin{proposition} \label{prop:K}
For each $\vep=0,1$, the intertwining relation \eqref{intertwine} has a solution 
if and only if
\[
\prod_{j\in I_\circ}\gamma_j=(-q)^n,
\]
in which case the solution is unique up to scalar multiple and given by
\[
K(x)v_\alpha = x^{\vep(\alpha_1-\alpha_{2n})}
\prod_{j=\vep,2+\vep,\ldots,2n-2+\vep}
(-q^{-1}\gamma_j)^{-\sum_{i=1+\vep}^j\alpha_i}
v^*_{\sigma^{(\vep)}(\alpha)}.
\]
\end{proposition}

\begin{proof}
In the proof we assume $i\in I_\bullet,j\in I_\circ$. Define $K_\alpha^\beta$ by
$K(x)v_\alpha=\sum_\beta K_\alpha^\beta v^*_\beta$. Note that $K_\alpha^\beta$ also
depends on $x$. Comparing the coefficients of $v^*_\beta$ in 
$K(x)\pi_{l,x}(a)v_\alpha=\pi^*_{l,x^{-1}}(a)K(x)v_\alpha$ with $k_i,e_i,f_i,b_j$  we obtain 
\begin{align}
&K_\alpha^\beta\ne0\quad\Rightarrow\quad \alpha_i-\alpha_{i+1}=-\beta_i+\beta_{i+1},
\label{(1)}\\
&[\beta_i+1]K_\alpha^{\beta+\eb_i-\eb_{i+1}}
=x^{2\delta_{i,0}}[\alpha_{i+1}]K_{\alpha+\eb_i-\eb_{i+1}}^{\beta}, \label{(2)}\\
&[\alpha_i+1]K_\alpha^{\beta+\eb_i-\eb_{i+1}}
=x^{2\delta_{i,0}}[\beta_{i+1}]K_{\alpha+\eb_i-\eb_{i+1}}^{\beta}, \label{(3)}\\
&x^{\delta_{j,0}}[\beta_{j+1}+1]K_\alpha^{\beta-\eb_j+\eb_{j+1}}
-x^{-\delta_{j,0}-\delta_{j,1}-\delta_{j,-1}}q^{-1}\gamma_j
[\beta_{j-1}+1]K_\alpha^{\beta+\eb_{j-1}-\eb_{j+2}} \nonumber\\
&\qquad=x^{-\delta_{j,0}}[\alpha_j]K_{\alpha-\eb_j+\eb_{j+1}}^{\beta}
-x^{\delta_{j,0}+\delta_{j,1}+\delta_{j,-1}}q^{-1}\gamma_j[\alpha_{j+2}]
K_{\alpha+\eb_{j-1}-\eb_{j+2}}^{\beta}.
\label{(4)}
\end{align}
Since we look for a nontrivial solution, we assume the right hand side of \eqref{(1)}.
This condition together with \eqref{(2)},\eqref{(3)} implies 
\begin{equation} \label{(5)}
\alpha_i=\beta_{i+1},\quad \beta_i=\alpha_{i+1}
\end{equation}
or equivalently $\beta=\sigma^{(\vep)}(\alpha)$. Then \eqref{(2)} or \eqref{(3)} reduces to
\begin{equation} \label{(a)}
K_\alpha^{\sigma^{(\vep)}(\alpha)}
=x^{2\delta_{i,0}}K_{\alpha+\eb_i-\eb_{i+1}}^{\sigma^{(\vep)}(\alpha+\eb_i-\eb_{i+1})}.
\end{equation}
Similarly, assuming \eqref{(5)}, \eqref{(4)} reduces to 
\begin{align*}
&x^{\delta_{j,0}}[\alpha_{j+2}](K_\alpha^{\beta-\eb_j+\eb_{j+1}}
+x^{\delta_{j,1}+\delta_{j,-1}}q^{-1}\gamma_jK_{\alpha+\eb_{j-1}-\eb_{j+2}}^{\beta})\\
&\qquad=x^{-\delta_{j,0}}[\alpha_j](K_{\alpha-\eb_j+\eb_{j+1}}^{\beta}
+x^{-\delta_{j,1}-\delta_{j,-1}}q^{-1}\gamma_jK_\alpha^{\beta+\eb_{j-1}-\eb_{j+2}}).
\end{align*}
If $\beta=\sigma^{(\vep)}(\alpha)+\eb_j-\eb_{j+1}$, the right hand side vanishes,
whereas if $\beta=\sigma^{(\vep)}(\alpha+\eb_j-\eb_{j+1})$, the left one does.
Under \eqref{(a)}, both conditions reduce to 
\[
K_{\alpha+\eb_{j-1}-\eb_{j+1}}^{\sigma^{(\vep)}(\alpha+\eb_{j-1}-\eb_{j+1})}/
K_\alpha^{\sigma^{(\vep)}(\alpha)}=-x^{\delta_{j,1}-\delta_{j,-1}}q^{-1}\gamma_j.
\]
Multiplying the above equation for $j=\vep,2+\vep,\ldots,2n-2+\vep$, we obtain 
the condition for $K$ to exist, and we obtain the unique solution up to scalar multiple.
\end{proof}

In view of this proposition, we set $\gamma_j=-q$ for any $j\in I_\circ$ later in this note.

\begin{theorem} \label{th:ref eq}
The reflection equation
\begin{align}\label{sin}
K_1(x)R^*((xy)^{-1})K_1(y)R(xy^{-1})
= 
R^{* *}(xy^{-1})K_1(y)R^*((xy)^{-1})K_1(x)
\end{align}
holds as a linear map $V_{l,x} \otimes V_{m,y} \rightarrow V^*_{l,x^{-1}} \otimes V^*_{m,y^{-1}}$.
Here $K_1(x)=K(x)\ot1$.
\end{theorem}

The proof is completely the same as that of Theorem 1 in \cite{KOYo} under the
assumption that $V_{l,x}\ot V_{m,y}$ is irreducible as a $\U^\iota_\vep$-module,
which is shown in next section.

\section{Proof of the irreducibility of $V_{l,x}\ot V_{m,y}$}

To show that the reflection equation holds (Theorem \ref{th:ref eq}), we need to prove
\begin{theorem} \label{th:irred}
As a ${\bf U}^\iota_\vep$-module, $V_{l,x}\ot V_{m,y}$ is irreducible.
\end{theorem}

Actually, even when the spectral parameters $x,y$ are specialized to 1, it is irreducible
as we will see below. Hence, in this section we set $x=y=1$, since it is enough to 
show the theorem. $V_{l,1}$ will be denoted by $V_l$. 
We can also restrict our proof to the $\vep=0$ case, since
the consideration for the $\vep=1$ case is just the repetition by shifting the index $i$
of the generators or the entries of $\alpha$. Finally, in view of Proposition 
\ref{prop:K}, we specialize $\gamma_i$ for $i\in I_\circ$ to be $-q$.

\subsection{Representation theory of $U_q(sl_2)$}

$U_q(sl_2)$ is the subalgebra of $\U$ generated only by $e_1,f_1,k_1$. Its irreducible 
representations are parametrized by their dimensions which run positive integers.
Let $U_l$ be the $(l+1)$-dimensional module of $U_q(sl_2)$. As a basis of $U_l$,
one can take $\{v_\alpha|\,|\alpha|=l\}$ in \eqref{obt1} with $n=1$. 
The actions of the generators $e_1,f_1,k_1$ are given by \eqref{otm1}-\eqref{otm3}.
It is well known that $U_l\ot U_m$ decomposes into $\min(l,m)+1$
components as
\[
U_l\ot U_m\simeq \bigoplus_{j=0}^{\min(l,m)}U_{l+m-2j}
\]
where a highest weight vector of $U_{l+m-2j}$ is given by 
\begin{equation} \label{w}
w^{(l,m)}_j=\sum_{p=0}^j(-1)^pq^{p(l-p+1)}{j\brack p}v_{(l-p,p)}\ot v_{(m-j+p,j-p)}.
\end{equation}
Here ${j\brack p}$ is the $q$-binomial coefficient defined by $\frac{[j]!}{[p]![j-p]!}$.

Now consider the subalgebra $\U(I_{\bullet})$ of $\U^\iota$ generated by 
$e_i,f_i,k_i\,(i\in I_\bullet)$. Recall $I_\bullet=\{1,3,\ldots,2n-1\}$.
$\U(I_{\bullet})$ is isomorphic to $U_q(sl_2)^{\ot n}$.
We want to construct a basis of $V_l\ot V_m$ using its $\U(I_\bullet)$-module 
structure. To parametrize the highest weight vectors of 
$V_l\ot V_m$, we introduce $n$-tuples of nonnegative integers 
$\lb=(l_1,\ldots,l_n),\mb=(m_1,\ldots,m_n)$ such that $|\lb|=l,|\mb|=m$.
Here we use the notation $|\lb|$ to signify the sum of entries of the vector $\lb$
irrespective of the number of entries. Let
\[
\iota:\bigoplus_{\lb,\mb}
(U_{l_1}\ot U_{m_1})\ot\cdots\ot(U_{l_n}\ot U_{m_n})\longrightarrow V_l\ot V_m
\]
be the linear map sending $(v_{(\alpha_1,\alpha_2)}\ot v_{(\beta_1,\beta_2)})\ot\cdots\ot
(v_{(\alpha_{2n-1},\alpha_{2n})}\ot v_{(\beta_{2n-1},\beta_{2n})})$ to $v_\alpha\ot v_\beta$.
Note that $U_{l_i}\ot U_{m_i}$ is the tensor product of the irreducible highest weight 
modules $U_{l_i},U_{m_i}$ of the $i$-th $U_q(sl_2)$ of $U_q(sl_2)^{\ot n}$ generated by 
$e_{2i-1},f_{2i-1},k_{2i-1}$. Since $U_q(sl_2)$ in different positions commute with
each other, one obtains the following proposition.

\begin{proposition} \label{prop:U(I_bullet)}
For any $\lb,\mb$ and $\jb=(j_1,\ldots,j_n)$ such that
$0\le j_i\le\min(l_i,m_i)$ for $1\le i\le n$,
\[
\wb^{(\lb,\mb)}_{\jb}
=\iota(w^{(l_1,m_1)}_{j_1}\ot\cdots\ot w^{(l_n,m_n)}_{j_n})
\]
is a $\U(I_\bullet)$-highest weight vector, and we have
$\bigoplus_{\lb,\mb,\jb}\U(I_\bullet)\wb^{(\lb,\mb)}_{\jb}=V_l\ot V_m$.
\end{proposition}

\subsection{Necessary formulas}

In what follows, we assume $i\in I_\circ=\{0,2,\ldots,2n-2\}$ and set $i=2s$.
By abuse of notation, we denote by $\eb_s$ ($s=1,\ldots,n$) the $s$-th standard basis vector of the $n$-dimensional space, although we have used it in section \ref{sec:Uq}
for the $2n$-dimensional space. $\eb_0$ should be understood as $\eb_n$.
For the action of $\U$ on the tensor product, we abbreviate $\Delta$.

\begin{proposition} \label{prop:bw}
On $V_{l}\ot V_{m}$, we have
\[
b_i\wb^{(\lb,\mb)}_{\jb}
=D'_1\wb^{(\lb-\eb_s+\eb_{s+1},\mb)}_{\jb-\eb_s}
+D'_2\wb^{(\lb,\mb-\eb_s+\eb_{s+1})}_{\jb-\eb_s}
+D'_3\wb^{(\lb+\eb_s-\eb_{s+1},\mb)}_{\jb-\eb_{s+1}}
+D'_4\wb^{(\lb,\mb+\eb_s-\eb_{s+1})}_{\jb-\eb_{s+1}},
\]
where 
\begin{align*}
D'_1&=-q^{-j_s-j_{s+1}+l_s+m_{s+1}+1}[j_s],\quad
D'_2=[j_s],\\
D'_3&=-q^{-j_s-j_{s+1}+l_{s+1}+m_{s+1}+1}[j_{s+1}],\quad
D'_4=q^{-2j_s-2j_{s+1}+l_s+l_{s+1}+2m_{s+1}+2}[j_{s+1}].
\end{align*}
\end{proposition}

\begin{proof}
Using Proposition \ref{prop:eb=be}, one finds that 
$b_i\wb^{(\lb,\mb)}_{\jb}$ is a $\U(I_\bullet)$-highest weight vector.
By the weight consideration, it should be a linear combination of the following vectors.
\[
\wb^{(\lb-\eb_s+\eb_{s+1},\mb)}_{\jb-\eb_s},\,
\wb^{(\lb,\mb-\eb_s+\eb_{s+1})}_{\jb-\eb_s},\,
\wb^{(\lb+\eb_s-\eb_{s+1},\mb)}_{\jb-\eb_{s+1}},\,
\wb^{(\lb,\mb+\eb_s-\eb_{s+1})}_{\jb-\eb_{s+1}}.
\]
The four coefficients can be calculated directly.
\end{proof}

\begin{proposition} \label{prop:bfw}
On $V_{l}\ot V_{m}$, we have
\begin{align*}
b_if_{i-1}\wb^{(\lb,\mb)}_{\jb}=\quad&\frac{[l_s+m_s-j_s+1]}{[l_s+m_s-2j_s+1]}(B'_1\wb^{(\lb-\eb_s+\eb_{s+1},\mb)}_{\jb}+B'_2\wb^{(\lb,\mb-\eb_s+\eb_{s+1})}_{\jb})\\
+&\frac{[j_{s+1}]}{[l_s+m_s-2j_s+1]}(B'_3\wb^{(\lb+\eb_s-\eb_{s+1},\mb)}_{\jb+\eb_s-\eb_{s+1}}+B'_4\wb^{(\lb,\mb+\eb_s-\eb_{s+1})}_{\jb+\eb_s-\eb_{s+1}})\\
+&\frac{[l_s+m_s-2j_s]}{[l_s+m_s-2j_s+1]}(D'_1f_{i-1}\wb^{(\lb-\eb_s+\eb_{s+1},\mb)}_{\jb-\eb_s}+D'_2f_{i-1}\wb^{(\lb,\mb-\eb_s+\eb_{s+1})}_{\jb-\eb_s}\\
+&D'_3f_{i-1}\wb^{(\lb+\eb_s-\eb_{s+1},\mb)}_{\jb-\eb_{s+1}}+D'_4f_{i-1}\wb^{(\lb,\mb+\eb_s-\eb_{s+1})}_{\jb-\eb_{s+1}}),\\
\end{align*}
\begin{align*}
b_if_{i+1}\wb^{(\lb,\mb)}_{\jb}=\quad&\frac{[j_s]}{[l_{s+1}+m_{s+1}-2j_{s+1}+1]}(C'_1\wb^{(\lb-\eb_s+\eb_{s+1},\mb)}_{\jb-\eb_s+\eb_{s+1}}+C'_2\wb^{(\lb,\mb-\eb_s+\eb_{s+1})}_{\jb-\eb_s+\eb_{s+1}})\\
+&\frac{[l_{s+1}+m_{s+1}-j_{s+1}+1]}{[l_{s+1}+m_{s+1}-2j_{s+1}+1]}(C'_3\wb^{(\lb+\eb_s-\eb_{s+1},\mb)}_{\jb}+C'_4\wb^{(\lb,\mb+\eb_s-\eb_{s+1})}_{\jb})\\
+&\frac{[l_{s+1}+m_{s+1}-2j_{s+1}]}{[l_{s+1}+m_{s+1}-2j_{s+1}+1]}(D'_1f_{i+1}\wb^{(\lb-\eb_s+\eb_{s+1},\mb)}_{\jb-\eb_s}+D'_2f_{i+1}\wb^{(\lb,\mb-\eb_s+\eb_{s+1})}_{\jb-\eb_s}\\
+&D'_3f_{i+1}\wb^{(\lb+\eb_s-\eb_{s+1},\mb)}_{\jb-\eb_{s+1}}+D'_4f_{i+1}\wb^{(\lb,\mb+\eb_s-\eb_{s+1})}_{\jb-\eb_{s+1}}),
\end{align*}
where
\begin{align*}
B'_1&=q^{j_s-j_{s+1}-m_s+m_{s+1}}[l_s-j_s],\quad
B'_2=[m_s-j_s],\\
B'_3&=-q^{j_s-j_{s+1}-l_s-m_s+l_{s+1}+m_{s+1}}[m_s-j_s],\quad
B'_4=-q^{2j_s-2j_{s+1}-l_s-2m_s+l_{s+1}+2m_{s+1}}[l_s-j_s],\\
C'_1&=-q^{-j_s+j_{s+1}+l_s-l_{s+1}}[m_{s+1}-j_{s+1}],\quad
C'_2=-[l_{s+1}-j_{s+1}],\\
C'_3&=q^{-j_s+j_{s+1}}[l_{s+1}-j_{s+1}],\quad
C'_4=q^{-2j_s+2j_{s+1}+l_s-l_{s+1}}[m_{s+1}-j_{s+1}],\\
\end{align*}
\end{proposition}

\begin{proof}
Using Proposition \ref{prop:eb=be}, we have
\begin{align*}
&e_{i-1}b_if_{i-1}\wb^{(\lb,\mb)}_{\jb}=b_ie_{i-1}f_{i-1}\wb^{(\lb,\mb)}_{\jb}=b_i\{k_{i-1}\}\wb^{(\lb,\mb)}_{\jb}=[l_s+m_s-2j_s]b_i\wb^{(\lb,\mb)}_{\jb},\\
&e_{i+1}b_if_{i+1}\wb^{(\lb,\mb)}_{\jb}=b_ie_{i+1}f_{i+1}\wb^{(\lb,\mb)}_{\jb}=b_i\{k_{i+1}\}\wb^{(\lb,\mb)}_{\jb}=[l_{s+1}+m_{s+1}-2j_{s+1}]b_i\wb^{(\lb,\mb)}_{\jb},
\end{align*}
where $\{k_i\}=\frac{k_i-k^{-1}_i}{q-q^{-1}}$ .
Thus same in Lemma \ref{prop:bw}, $e_{i+1}b_if_{i+1}$ and $e_{i-1}b_if_{i-1}$ are a linear combination of the following vectors.
\[
\wb^{(\lb-\eb_s+\eb_{s+1},\mb)}_{\jb-\eb_s},\,
\wb^{(\lb,\mb-\eb_s+\eb_{s+1})}_{\jb-\eb_s},\,
\wb^{(\lb+\eb_s-\eb_{s+1},\mb)}_{\jb-\eb_{s+1}},\,
\wb^{(\lb,\mb+\eb_s-\eb_{s+1})}_{\jb-\eb_{s+1}}.
\]
By considering weight, one find that $b_if_{i-1}$ and  $b_if_{i+1}$ are a linear combination like a assertion, and coefficients can be calculated directly.
\end{proof}

\begin{corollary} \label{cor:bfw}
On $V_{l}\ot V_{m}$, we have
\begin{align*}
b_if_{i+1}\wb^{(\lb,\mb)}_{\ob}&=
[l_{s+1}]\wb^{(\lb+\eb_s-\eb_{s+1},\mb)}_{\ob}
+q^{l_s-l_{s+1}}[m_{s+1}]\wb^{(\lb,\mb+\eb_s-\eb_{s+1})}_{\ob},\\
b_if_{i-1}\wb^{(\lb,\mb)}_{\ob}
&=q^{m_{s+1}-m_s}[l_s]\wb^{(\lb-\eb_s+\eb_{s+1},\mb)}_{\ob}
+[m_s]\wb^{(\lb,\mb-\eb_s+\eb_{s+1})}_{\ob}.
\end{align*}
\end{corollary}

\begin{proposition} \label{prop:bffw}
On $V_l\ot V_m$, we have
\begin{align}
&b_if_{i-1}f_{i+1}\wb^{(\lb,\mb)}_{\jb} \nonumber\\
&=A_1\wb^{(\lb-\eb_s+\eb_{s+1},\mb)}_{\jb+\eb_{s+1}}
+B_1f_{i+1}\wb^{(\lb-\eb_s+\eb_{s+1},\mb)}_{\jb}
+C_1f_{i-1}\wb^{(\lb-\eb_s+\eb_{s+1},\mb)}_{\jb-\eb_s+\eb_{s+1}}
+D_1f_{i-1}f_{i+1}\wb^{(\lb-\eb_s+\eb_{s+1},\mb)}_{\jb-\eb_s} \nonumber\\
&+A_2\wb^{(\lb,\mb-\eb_s+\eb_{s+1})}_{\jb+\eb_{s+1}}
+B_2f_{i+1}\wb^{(\lb,\mb-\eb_s+\eb_{s+1})}_{\jb}
+C_2f_{i-1}\wb^{(\lb,\mb-\eb_s+\eb_{s+1})}_{\jb-\eb_s+\eb_{s+1}}
+D_2f_{i-1}f_{i+1}\wb^{(\lb,\mb-\eb_s+\eb_{s+1})}_{\jb-\eb_s} \nonumber\\
&+A_3\wb^{(\lb+\eb_s-\eb_{s+1},\mb)}_{\jb+\eb_s}
+B_3f_{i+1}\wb^{(\lb+\eb_s-\eb_{s+1},\mb)}_{\jb+\eb_s-\eb_{s+1}}
+C_3f_{i-1}\wb^{(\lb+\eb_s-\eb_{s+1},\mb)}_{\jb}
+D_3f_{i-1}f_{i+1}\wb^{(\lb+\eb_s-\eb_{s+1},\mb)}_{\jb-\eb_{s+1}} \nonumber\\
&+A_4\wb^{(\lb,\mb+\eb_s-\eb_{s+1})}_{\jb+\eb_s}
+B_4f_{i+1}\wb^{(\lb,\mb+\eb_s-\eb_{s+1})}_{\jb+\eb_s-\eb_{s+1}}
+C_4f_{i-1}\wb^{(\lb,\mb+\eb_s-\eb_{s+1})}_{\jb}
+D_4f_{i-1}f_{i+1}\wb^{(\lb,\mb+\eb_s-\eb_{s+1})}_{\jb-\eb_{s+1}}, \label{16terms}
\end{align}
where
\begin{align*}
A_1&=q^{j_s+j_{s+1}-l_{s+1}-m_s-1}\frac{[l_s-j_s][m_{s+1}-j_{s+1}][l_s+m_s-j_s+1]}
{[l_s+m_s-2j_s+1][l_{s+1}+m_{s+1}-2j_{s+1}+1]}\\
A_2&=-\frac{[l_{s+1}-j_{s+1}][m_s-j_s][l_s+m_s-j_s+1]}
{[l_s+m_s-2j_s+1][l_{s+1}+m_{s+1}-2j_{s+1}+1]}\\
A_3&=q^{j_s+j_{s+1}-l_s-m_s-1}
\frac{[l_{s+1}-j_{s+1}][m_s-j_s][l_{s+1}+m_{s+1}-j_{s+1}+1]}
{[l_s+m_s-2j_s+1][l_{s+1}+m_{s+1}-2j_{s+1}+1]}\\
A_4&=-q^{2j_s+2j_{s+1}-l_s-l_{s+1}-2m_s-2}
\frac{[l_s-j_s][m_{s+1}-j_{s+1}][l_{s+1}+m_{s+1}-j_{s+1}+1]}
{[l_s+m_s-2j_s+1][l_{s+1}+m_{s+1}-2j_{s+1}+1]},
\end{align*}
\begin{align*}
B_j&=B'_j\frac{[l_s+m_s-j_s+1][l_{s+1}+m_{s+1}-2j_{s+1}]}{[l_s+m_s-2j_s+1][l_{s+1}+m_{s+1}-2j_{s+1}+1]}\quad(j=1,2),\\
&=B'_j\frac{[j_{s+1}][l_{s+1}+m_{s+1}-2j_{s+1}]}{[l_s+m_s-2j_s+1][l_{s+1}+m_{s+1}-2j_{s+1}+1]}\quad(j=3,4),\\
C_j&=C'_j\frac{[j_s][l_s+m_s-2j_s]}{[l_s+m_s-2j_s+1][l_{s+1}+m_{s+1}-2j_{s+1}+1]}\quad(j=1,2),\\
&=C'_j\frac{[l_s+m_s-2j_s][l_{s+1}+m_{s+1}-j_{s+1}+1]}{[l_s+m_s-2j_s+1][l_{s+1}+m_{s+1}-2j_{s+1}+1]}\quad(j=3,4),\\
D_j&=D'_j\frac{[l_s+m_s-2j_s][l_{s+1}+m_{s+1}-2j_{s+1}]}{[l_s+m_s-2j_s+1][l_{s+1}+m_{s+1}-2j_{s+1}+1]}\quad(j=1,2,3,4).
\end{align*}
\end{proposition}

\begin{proof}
Similar to Proposition\ref{prop:bw} and \ref{prop:bfw}, $b_if_{i-1}f_{i+1}\wb^{(\lb,\mb)}_{\jb}$ can be expressed with suitable scalars $A_j,B_j,C_j,D_j$ ($1\le j\le 4$) as \eqref{16terms}.
By applying $e_{i-1}e_{i+1}$ on both sides, the first to third terms in each
line of the right hand side vanish. So by Proposition \ref{prop:bw}, $D_j$ ($1\le j\le 4$) is determined.
Then, by applying $e_{i+1}$ on both sides of \eqref{16terms}, $B_j$ ($1\le j\le 4$) is determined, and
by applying $e_{i-1}$, $C_j$ ($1\le j\le 4$) is done by Proposition \ref{prop:bfw}.
Finally, $A_j$ ($1\le j\le 4$) is determined by a direct calculation.
\end{proof}

\begin{corollary} \label{cor:bffw}
On $V_{l}\ot V_{m}$, we have
\begin{equation*}
\begin{split}
b_if_{i-1}f_{i+1}\wb^{(\lb,\mb)}_{\jb}
=&a_1\wb^{(\lb-\eb_s+\eb_{s+1},\mb)}_{\jb+\eb_{s+1}}
+a_2\wb^{(\lb,\mb-\eb_s+\eb_{s+1})}_{\jb+\eb_{s+1}}
+a_3\wb^{(\lb+\eb_s-\eb_{s+1},\mb)}_{\jb+\eb_s}
+a_4\wb^{(\lb,\mb+\eb_s-\eb_{s+1})}_{\jb+\eb_s}\\
&+(\text{other terms}),
\end{split}
\end{equation*}
where $a_j$ $(j=1,2,3,4)$ is given in Proposition \ref{prop:bffw} and $($other terms$)$ stands for the linear combination of vectors of the form
$\wb^{(\lb',\mb')}_{\jb'}$ possibly applied by $f_{i-1},f_{i+1}$
with $(\lb',\mb')$ appearing in the right hand side and 
$j'_k\le j_k$ for $1\le k\le n$.
\end{corollary}

\subsection{Proof of Theorem \ref{th:irred}}

We prove Theorem \ref{th:irred} when $\vep=0$. Suppose $W$ is a nonzero 
$\U^\iota$-invariant subspace of $V_l\ot V_m$. Note that $\U^\iota$ contains 
$\U(I_\bullet)$. In view of Proposition \ref{prop:U(I_bullet)}, 
one can assume that $W$ contains a vector of the form
\begin{equation} \label{lin comb}
\sum_{\lb,\mb,\jb}c(\lb,\mb,\jb)\wb^{(\lb,\mb)}_{\jb}
\end{equation}
where $c(\lb,\mb,\jb)\in\Q(q)$ and $\lb,\mb,\jb$ run over all possible integer vectors
such that $l_s+m_s-2j_s$ is constant for any $s=1,\ldots,n$. By applying $b_i$ 
($i\in I_\circ$) in a suitable order, from Proposition \ref{prop:bw} one can assume
$\jb=\ob$ in \eqref{lin comb}. Then by Corollary \ref{cor:bfw}, one can eventually
assume $\lb=l\eb_1,\mb=m\eb_1$ where $l=|\lb|,m=|\mb|$. Hence, we have
$\wb^{(l\eb_1,m\eb_1)}_{\ob}\in W$.

Next show $\wb^{(l_1\eb_1+l_2\eb_2,m_1\eb_1+m_2\eb_2)}_{\ob}\in W$ for any 
$l_1,l_2,m_1,m_2$ such that $l_1+l_2=l,m_1+m_2=m$.
We do it by induction on $k=l_2+m_2$. The $k=0$ case is done. Assume 
$\wb^{(l_1\eb_1+l_2\eb_2,m_1\eb_1+m_2\eb_2)}_{\ob}\in W$ for $l_2+m_2=k$.
By Corollary \ref{cor:bfw}, we have
\begin{align*}
b_2f_1\wb^{(l_1\eb_1+l_2\eb_2,m_1\eb_1+m_2\eb_2)}_{\ob}
&=q^{m_2-m_1}[l_1]\wb^{((l_1-1)\eb_1+(l_2+1)\eb_2,m_1\eb_1+m_2\eb_2)}_{\ob}
+[m_1]\wb^{(l_1\eb_1+l_2\eb_2,(m_1-1)\eb_1+(m_2+1)\eb_2)}_{\ob},\\
(b_4f_5)\cdots(b_{2n-2}f_{2n-1})(b_0f_1&)\wb^{(l_1\eb_1+l_2\eb_2,m_1\eb_1+m_2\eb_2)}_{\ob}\\
&=[l_1]\wb^{((l_1-1)\eb_1+(l_2+1)\eb_2,m_1\eb_1+m_2\eb_2)}_{\ob}+q^{l_2-l_1}[m_1]\wb^{(l_1\eb_1+l_2\eb_2,(m_1-1)\eb_1+(m_2+1)\eb_2)}_{\ob}.
\end{align*}
If $l_1+m_1\neq l_2+m_2$, these two vectors are linearly independent. 
Hence the induction proceeds up to $k\le l_1+m_1$. When $l_2+m_2\ge l_1+m_1$,
we first recognize that $\wb^{(l\eb_2,m\eb_2)}_{\ob}\in W$ by applying $(b_2f_1)^{l+m}$
to $\wb^{(l\eb_1,m\eb_1)}_{\ob}$. We then do the same exercise as before.

Let us now show $W$ contains $\wb^{(\lb,\mb)}_{\ob}$ for any possible $\lb$ and $\mb$.
From the previous paragraph, we know 
$\wb^{(l_1\eb_1+l_2\eb_2,m\eb_1)}_{\ob}\in W$. Applying $b_if_{i-1}$ ($i=4,\ldots,2n-2$)
suitable times, we know $\wb^{(\lb,m\eb_1)}_{\ob}\in W$ for any $\lb$. Then by doing
similarly including $i=2$, we know $\wb^{(\lb,\mb)}_{\ob}\in W$ for any $\lb,\mb$.

By Proposition \ref{prop:U(I_bullet)}, it is enough to show $W$ contains 
$\wb^{(\lb,\mb)}_{\jb}$ for any possible $\lb,\mb,\jb$. From the considerations so far,
it is true when $|\jb|=0$. The following proposition makes the induction on $|\jb|$ work
and finishes the proof of Theorem \ref{th:irred}.

\begin{proposition} \label{prop:fr}
Consider the following matrix $C$ depending on $l,m,j$. 
Its row index runs over all $(i,\lb,\mb,
\jb)$ with $i=0,2,\ldots,2n-2$ and $|\lb|=l,|\mb|=m,|\jb|=j$, and
its column index runs over all $(\lb',\mb',\jb')$ with 
$|\lb'|=l,|\mb'|=m,|\jb'|=j+1$. The entry for the pair
$((i,\lb,\mb,\jb),(\lb',\mb',\jb'))$ is given by 
the coefficient of $\wb^{(\lb',\mb')}_{\jb'}$ in 
$b_if_{i-1}f_{i+1}\wb^{(\lb,\mb)}_{\jb}$ in the previous
proposition. Then $C$ is of full rank. Note that the rank does not depend on the orders
of the index sets.
\end{proposition}

\begin{proof}
Let $A$ be the subring of $\Q(q)$ defined by 
$A=\{f(q)\in\Q(q)\,|\,\text{$f(q)$ is regular at $q=0$}\}$. 
Let $\alpha_t$ ($t=1,2,3,4$) be the largest integer such that $a_t$ in Corollary \ref{cor:bffw}
belongs to $q^{\alpha_t}A$. We have
\begin{align*}
\alpha_1-\alpha_2&=\alpha_4-\alpha_3=j_s+j_{s+1}-l_s-m_{s+1}-1<0,\\
\alpha_4-\alpha_1&=2j_{s+1}-l_{s+1}-m_{s+1}-1<0,
\end{align*}
since $j_t\le\min(l_t,m_t)$ ($t=s,s+1$).  
Therefore, $\alpha_4$ is minimal and the others are strictly larger.

For $\wb^{(\lb',\mb')}_{\jb'}$ such that $|\lb'|=l,|\mb'|=m,|\jb'|=j+1$, choose the minimal
$s$ such $j'_s>0$ and consider 
$b_if_{i-1}f_{i+1}\wb^{(\lb',\mb'-\eb_s+\eb_{s+1})}_{\jb'-\eb_s}$ with $i=2s$.
By Proposition \ref{prop:bffw} the fourth term of the above is nonzero. 
Consider the row of $C$ corresponding to the index 
$(i,\lb',\mb'-\eb_s+\eb_{s+1},\jb'-\eb_s)$. By multiplying a suitable scalar to this row,
one can make the $((i,\lb',\mb'-\eb_s+\eb_{s+1},\jb'-\eb_s),(\lb',\mb',\jb'))$-entry
of $C$ be 1, and the other three nonzero entries in the same row belong to 
$qA$. Consider the square matrix $C'$ obtained by varying 
all possible $(\lb',\mb',\jb')$ and picking the corresponding renormalized rows.
Then from the construction, $\det C'$ belongs to $\{\pm1\}+qA$.
Hence the assertion is confirmed.
\end{proof}

\section*{Acknowledgments}
The authors thank Atsuo Kuniba, Hideya Watanabe, Yasuhiko Yamada and Akihito
Yoneyama for comments and giving us references.
M.O. is supported by Grants-in-Aid for Scientific Research No.~19K03426
and No.~16H03922 from JSPS.
This work was partly supported by Osaka City University Advanced Mathematical Institute (MEXT Joint Usage/Research Center on Mathematics and Theoretical Physics JPMXP0619217849).

\end{document}